\documentclass[11pt]{amsart}
\usepackage{amscd}
\usepackage{amsmath,empheq}
\usepackage{amsfonts}
\usepackage{amssymb}
\usepackage{mathrsfs}

\usepackage{amssymb,amsmath,comment,amsthm}
\catcode`\@=11 \@addtoreset{equation}{section}
\def\thesection{\arabic{section}}

\def\theequation{\thesection.\arabic{equation}}
\catcode`\@=12
\usepackage{colortbl}%
\usepackage{a4wide}
\usepackage{parskip}
\newcommand{\ds} {\displaystyle}
\newcommand{\e}{\epsilon}

\newcommand{\al} {\alpha}
\newcommand{\ba} {\beta}
\newcommand{\ga} {\gamma}
\newcommand{\Ga} {\Gamma}
\newcommand{\Om} {\Omega}
\newcommand{\ra} {\rightarrow}

\newcommand{\de} {\delta}
\newcommand{\De} {\Delta}
\newcommand{\la} {\lambda}

\newcommand{\noi} {\noindent}

\newcommand{\mb} {\mathbb}
\newcommand{\mc} {\mathcal}

\usepackage[all]{xy}
\catcode`\@=11
\def\theequation{\@arabic{\c@section}.\@arabic{\c@equation}}
\catcode`\@=12

\def\QED{\hfill {$\square$}\goodbreak \medskip}

\newtheorem{Theorem}{Theorem}[section]
\newtheorem{Lemma}[Theorem]{Lemma}
\newtheorem{Proposition}[Theorem]{Proposition}
\newtheorem{Corollary}[Theorem]{Corollary}
\newtheorem{Remark}[Theorem]{Remark}
\newtheorem{Definition}[Theorem]{Definition}

\begin{document}
{\vspace{0.01in}
\title[Unbalanced fractional elliptic problems]
{\small Unbalanced fractional elliptic problems  with exponential nonlinearity: subcritical and critical cases}

\author[D. Kumar]{Deepak Kumar}
\address[D. Kumar]{Department of Mathematics, Indian Institute of Technology Delhi,
	Hauz Khaz, New Delhi-110016, India}
\email{\tt deepak.kr0894@gmail.com}

\author[V.D. R\u{a}dulescu]{Vicen\c{t}iu D. R\u{a}dulescu}
\address[V.D. R\u{a}dulescu]{Faculty of Applied Mathematics, AGH University of Science and Technology, al. Mickiewicza 30, 30-059 Krak\'ow, Poland \& Department of Mathematics, University of Craiova, 200585 Craiova, Romania}
\email{\tt radulescu@inf.ucv.ro}

\author[K. Sreenadh]{K. Sreenadh}
\address[K. Sreenadh]{Department of Mathematics, Indian Institute of Technology Delhi,
	Hauz Khaz, New Delhi-110016, India}
\email{\tt sreenadh@gmail.com}

%\author{ {\bf Deepak Kumar\footnote{deepak.kr0894@gmail.com} \; and \;  K. Sreenadh\footnote{
%			e-mail: sreenadh@gmail.com}} \\ Department of Mathematics,\\ Indian Institute of %Technology Delhi,\\	Hauz Khaz, New Delhi-110016, India. }

\date{}


\begin{abstract}
This paper deals with the qualitative analysis of solutions
 to the following $(p,q)$-fractional equation:
\begin{equation*}
\begin{array}{rllll}
(-\Delta)^{s_1}_{p}u+(-\Delta)^{s_2}_{q}u+V(x) \big(|u|^{p-2}u+|u|^{q-2}u\big)  = K(x)\frac{f(u)}{|x|^\ba} \;  \text{ in } \mb R^N,
\end{array}
\end{equation*}
\noi where $1< q< p$, $0<s_2\leq s_1<1$, $ps_1=N$, $\ba\in[0,N)$, and $V,K:\mb R^N\to\mb R$, $f:\mb R\to \mb R$ are continuous functions satisfying some natural hypotheses. We 
are concerned both with the case when $f$ has a subcritical growth and with the critical framework with respect to the exponential nonlinearity. By combining a
Moser-Trudinger
type inequality for fractional Sobolev spaces with Schwarz symmetrization techniques and related variational methods, we
prove the existence of  nonnegative solutions.
\end{abstract}

\maketitle

%\bigskip
%\vfill\eject

\section {Introduction}
\noindent In this paper, we are concerned with the study of a nonlinear nonlocal problem whose features are the following: (a) the presence of several differential operators with different growth, which generates a {\it double phase} associated energy; (b) the reaction combines the multiple effects generated by a Hardy singular potential and a term with subcritical or critical growth with respect to the exponential nonlinearity; (c) due to the unboundedness of the domain, Cerami sequences do not have the compactness property;  (d)
we overcome the lack of compactness by exploiting the special properties of the associated potential; (e) the proofs combine refined techniques, including a Moser-Trudinger type inequality for fractional
Sobolev spaces and  Schwarz symmetrization tools. 
 Summarizing, this paper
 is concerned with the refined qualitative and bifurcation analysis of solutions for a class of {\it singular} nonlocal problems driven by differential operators with {\it unbalanced growth}. The arguments cover both the subcritical and critical cases.
 
  We recall in what follows some of the outstanding contributions of the Italian school  to the study of unbalanced integral functionals and double phase problems.  We first refer to the pioneering contributions of Marcellini \cite{marce1,marce2,marce3} who studied lower semicontinuity and regularity properties of minimizers of certain quasiconvex integrals. Problems of this type arise in nonlinear elasticity and are connected with the deformation of an elastic body, cf. Ball \cite{ball1,ball2}. We also refer to Fusco and Sbordone \cite{fusco} for the study of regularity of minima of anisotropic integrals.

In order to recall the roots of double phase problems, let us assume that $\Omega$ is a bounded domain in ${\mathbb R}^n$ ($N\geq 2$) with smooth boundary. If $u:\Omega\to{\mathbb R}^n$ is the displacement and if $Du$ is the $n\times n$  matrix of the deformation gradient, then the total energy can be represented by an integral of the type
\begin{equation}\label{paolo}I(u)=\int_{\Omega} f(x,Du(x))dx,\end{equation}
where the energy function $f=f(x,\xi):\Omega\times{\mathbb R}^{n\times n}\to{\mathbb R}$ is quasiconvex with respect to $\xi$. One of the simplest examples considered by Ball is given by functions $f$ of the type
$$f(\xi)=g(\xi)+h({\rm det}\,\xi),$$
where ${\rm det}\,\xi$ is the determinant of the $n\times n$ matrix $\xi$, and $g$, $h$ are nonnegative convex functions, which satisfy the growth conditions
$$g(\xi)\geq c_1\,|\xi|^p;\quad\lim_{t\to+\infty}h(t)=+\infty,$$
where $c_1$ is a positive constant and $1<p<n$. The condition $p< n$ is necessary to study the existence of equilibrium solutions with cavities, that is, minima of the integral \eqref{paolo} that are discontinuous at one point where a cavity forms; in fact, every $u$ with finite energy belongs to the Sobolev space $W^{1,p}(\Omega,{\mathbb R}^n)$, and thus it is a continuous function if $p>n$. In accordance with these problems arising in nonlinear elasticity, Marcellini \cite{marce1,marce2} considered continuous functions $f=f(x,u)$ with {\it unbalanced growth} that satisfy
$$c_1\,|u|^p\leq |f(x,u)|\leq c_2\,(1+|u|^q)\quad\mbox{for all}\ (x,u)\in\Omega\times{\mathbb R},$$
where $c_1$, $c_2$ are positive constants and $1\leq p\leq q$. Regularity and existence of solutions of elliptic equations with $p,q$--growth conditions were studied in \cite{marce2}.

The study of non-autonomous functionals characterized by the fact that the energy density changes its ellipticity and growth properties according to the point has been continued in a series of remarkable papers by Mingione {\it et al.} \cite{baroni1, baroni, beck, colombo0, colombo1}. These contributions are in relationship with the works of Zhikov \cite{zhikov1}, in order to describe the
behavior of phenomena arising in nonlinear
elasticity.
In fact, Zhikov intended to provide models for strongly anisotropic materials in the context of homogenisation.
 In particular, Zhikov considered the following model of
functional  in relationship to the Lavrentiev phenomenon:
$$
{\mathcal P}_{p,q}(u) :=\int_{\Omega} (|\nabla u|^p+a(x)|\nabla u|^q)dx,\quad 0\leq a(x)\leq L,\ 1<p<q.
$$
In this functional, the modulating coefficient $a(x)$ dictates the geometry of the composite made by
two differential materials, with hardening exponents $p$ and $q$, respectively.

The functional ${\mathcal P}_{p,q}$ falls in the realm of the so-called functionals with
nonstandard growth conditions of $(p, q)$--type, according to Marcellini's terminology. This is a functional of the type in \eqref{paolo}, where the energy density satisfies
$$|\xi|^p\leq f(x,\xi)\leq  |\xi|^q+1,\quad 1\leq p\leq q.$$

Another significant model example of a functional with $(p,q)$--growth studied by Mingione {\it et al.} is given by
$$u\mapsto \int_{\Omega} |\nabla u|^p\log (1+|\nabla u|)dx,\quad p\geq 1,$$
which is a logarithmic perturbation of the $p$-Dirichlet energy.

Recent contributions to the study of double phase problems can be found in \cite{alves, bahrouni, kumar, papa1, papa2, papa3} (local case) and \cite{ambrosio, goel, zhang0} (nonlocal case).

\section{Statement of the problem and abstract setting}
In this paper, we are concerned with the existence of solutions for the following singular $(p,q)$-fractional equation:
\begin{equation*}
  \noindent(\mc P)\qquad \begin{array}{rllll}
  (-\Delta)^{s_1}_{p}u+(-\Delta)^{s_2}_{q}u+V(x) \big(|u|^{p-2}u+|u|^{q-2}u\big)  = K(x)\frac{f(u)}{|x|^\ba} \;  \text{ in } \mb R^N,
  \end{array}
 \end{equation*}
\noi where $1< q< p$, $0<s_2\leq s_1<1$, $2\leq N=ps_1$, $\ba\in[0,N)$, and $V,K:\mb R^N\to\mb R$, $f:\mb R\to \mb R$ are continuous functions  satisfying some natural assumptions. Let $(-\Delta)^{s}_{p}$ denote the fractional $p$-Laplace operator defined as
\begin{equation*}
{(-\Delta)^{s}_pu(x)}= 2\ds\lim_{\e\ra 0}\int_{\mb R^N\setminus B_\e(x)} \frac{|u(x)-u(y)|^{p-2}(u(y)-u(x))}{|x-y|^{N+ps}}dy.
\end{equation*}
Problems involving the fractional Laplacian, as in $(\mc P)$, arise from a wide range of real world applications such as optimization, phase transition, anomalous diffusion, image processing, soft thin films, conservation laws and water waves, for a list of more bibliography and other details on this topic we refer to \cite{nezzaH}.  The main motivation to study problems with leading operators given in $(\mc P)$ comes when $s_1=s_2=1$, which is the local case. Here the leading operator,  known as $(p,q)$-Laplacian,  arises while studying the stationary solutions of general reaction-diffusion equation
\begin{equation}\label{prb1}
u_t= \mathrm{div} [A(u)\nabla u]+ r(x,u),
\end{equation}
where $A(u)= |\nabla u|^{p-2}+|\nabla u|^{q-2}$. 

Problem \eqref{prb1} has  applications to biophysics, plasma physics and chemical reactions, where $u$ corresponds to the concentration term, the first term on the right-hand side represents diffusion with a diffusion coefficient $A(u)$ and the second term is the reaction, which relates to sources and loss processes. For more details,  readers are referred to \cite{marano}.

In the local case, that is, when $s_1=s_2=1$, problem $(\mc P)$ is motivated by the famous Moser-Trudinger inequality. This comes into the picture because of the fact that $W^{1,N}(\mb R^N)$ is embedded into $L^p(\mb R^N)$ for all $N\le p<\infty$ but not in $L^\infty(\Om)$, hence in this case the critical nonlinearity is considered to have exponential type growth condition. These kinds of problems were studied by several authors, see for instance, \cite{adi,desoz,giaco}.
%on a bounded domain. Subsequently, an ample amount of research has been done in this direction, among them we mention here
As far as problems with singular exponential nonlinearity is concerned, Adimurthi and Sandeep \cite{adisand}  proved that the embedding $W^{1,N}_0(\Om)\ni u\mapsto |x|^{-\ba} e^{\al |u|^{N/(N-1)}} \in L^1(\Om)$ is compact if $\frac{\al}{\al_N}+\frac{\ba}{N}<1$ and is continuous if  $\frac{\al}{\al_N}+\frac{\ba}{N}=1$. Using this result they studied problems having singular exponential type nonlinearity in a bounded domain. In the case of $\mb R^N$, Adimuthi and Yang \cite{adiyan} considered the following singular problem
\begin{align*}
	-\Delta_N u+V(x)|u|^{N-2}u = \frac{f(x,u)}{|x|^\ba}+\e h(x) \;  \text{ in } \mb R^N,
\end{align*}
where among other assumptions, $f$ has exponential growth condition and $h$ is in the dual of $W^{1,N}(\mb R^N)$. Here authors established singular Moser-Trudinger type inequality for whole $\mb R^N$ and obtained the existence result for a mountain pass solution when $\e>0$ is small. Subsequently, Yang \cite{yang} and Goyal and Sreenadh \cite{goyal} studied similar singular problems in the whole of $\mb R^N$. In the latter work, authors proved the existence and multiplicity results using Nehari manifold method. 

Regarding the problems involving operators with unbalanced growth conditions, we mention the recent work of Figueiredo and Nunes \cite{figu}. Using the method of Nehari manifold authors proved the existence of a solution for $(N,p)$ type equations in bounded domains. In \cite{fiscel}, Fiscella and Pucci studied the following $(N,p)$ equation:
\begin{align*}
-\De_p u-\De_N u+|u|^{p-2}u +|u|^{N-2}u = \la h(x)u_+^{q-1} +\ga f(x,u) \; \mbox{ in }\mb R^N,
\end{align*}
where $1<q,p<N<\infty$, $N\ge 2$, $h(x)\ge 0$, $\la,\ga>0$ are parameters and the function $f$ has exponential type growth condition. In this work authors proved the existence of multiple solutions for small $\la>0$ and large $\ga$.

In the nonlocal setting, we mention the work of Giacomoni {\it et al.} \cite{giaco1}. Here, the authors proved existence of multiple solutions using Nehari manifold for the $1/2$-Laplacian problem in a bounded domain of $\mb R$.
Subsequently, Zhang \cite{zhang} established Moser-Trudinger type inequality in fractional Sobolev-Slobodeckij spaces $W^{s,p}(\mb R^N)$ and proved existence and multiplicity of solutions for the following fractional Laplacian equation
\begin{align*}
 (-\Delta)_p^s u+V(x)|u|^{p-2}u = f(x,u)+\e h(x) \;  \text{ in } \mb R^N,
\end{align*}
 when $\e>0$ is sufficiently small.
% For higher dimensional case, Perera and Squassina \cite{perera} obtained results for the following problem
%\begin{align*}
%(-\De)^s_{N/s} u= \la |u|^{N/(N-2s)} \exp\{|u|^{N/(N-s)} \} \; \mbox{ in }\Om, \quad u=0 \; \mbox{ in }\mb R^N\setminus\Om,
%\end{align*}
%where $\la>0$ is a parameter.
Recently, Mingqi {\it et al.} \cite{ming} and Xiang {\it et al.} \cite{xiang} studied fractional Kirchhoff problems with exponential nonlinearity in bounded domain and in $\mb R^N$, respectively.

Problems of the type $(\mc P)$ involving potential $K$ and exponential type nonlinearity have been studied by do \'O {\it et al.} \cite{doMySq} for the case $N=1$ and $s=1/2$. In this work, authors considered the following problem:
\begin{align*}
	(-\Delta)^{1/2} u+ u = K(x)g(u) \;  \text{ in } \mb R.
\end{align*}
Under certain conditions on $K$, authors proved compactness results, which was absent due to unboundedness of the domain, and obtained existence of a nontrivial nonnegative solution in the cases when $g$ possesses subcritical or critical growth condition. Subsequently, this work was generalized by Miyagaki and Pucci \cite{miyPu} for Kirchoff problem in $1$-dimension.

\section{Main results: subcritical and critical cases}
The main purpose in the present paper is to obtain the existence of nontrivial nonnegative solutions to $(\mc P)$ under the following assumptions on  $V,K:\mb R^N\to \mb R$.
  \begin{enumerate}
 	\item[(i)] The function $V$ is continuous and there exists a constant $V_0>0$ such that $V(x)\ge V_0>0$ for all $x\in\mb R^N$. 
 	\item[(ii)] The function $K\in C(\mb R^N) \cap L^\infty(\mb R^N)$ and is positive in $\mb R^N$.
 	\item[(iii)] For any sequence $\{A_n \}$ of measurable sets of $\mb R^N$ with $|A_n|\le R$, for all $n\in\mb N$ and some $R>0$, the following holds 
 	\begin{align}\label{eqK}
 	\lim_{r\ra\infty} \int_{A_n\cap B_r^c(0)} K(x)dx =0 \quad \mbox{uniformly w.r.t. } n\in\mb N.
 	\end{align}
 \end{enumerate}
To define the natural space which contains all the solutions of problem $(\mc P)$, we first recall the notion of following spaces. For $1<p<\infty$ and $0<s<1$, the fractional Sobolev space is defined as
 \begin{align*}
      W^{s,p}(\mb R^N)= \left\lbrace u \in L^p(\mb R^N): [u]_{s,p}<+ \infty \right\rbrace
 \end{align*}
 \noi endowed with the norm  $\|u\|_{W^{s,p}(\mb R^N)}=  \|u\|_{L^p(\mb R^N)}+ [u]_{s,p}$, where
 \begin{align*}
 [u]_{s,p}^p =\int_{\mb R^N}\int_{\mb R^N} \frac{|u(x)-u(y)|^p}{|x-y|^{N+sp}}dxdy.
 \end{align*}
Let $\widetilde{W}_V^{s_1,p}(\mb R^N)$ be the space defined as
\[ \widetilde{W}_V^{s_1,p}(\mb R^N):= \bigg\{ u\in W^{s_1,p}(\mb R^N) : \int_{\mb R^N} V(x)|u(x)|^p dx<\infty \bigg\}, \]
which is a reflexive Banach space when endowed with the norm
 \begin{align*}
 	\| u\|_{s_1,p} = \left( [u]_{s_1,p}^p+ \int_{\mb R^N} V(x)|u(x)|^p dx \right)^{1/p}
 \end{align*}
and analogously we define $\widetilde{W}_V^{s_2,q}(\mb R^N)$. From \cite{nezzaH,pucci}, we have the following continuous embedding result
 \begin{align}\label{emd}
 \widetilde{W}_V^{s_1,p}(\mb R^N)\hookrightarrow W^{s_1,p}(\mb R^N)\hookrightarrow L^m(\mb R^N), \; \mbox{ for all }m\ge p.
 \end{align}
Let $X:= \widetilde{W}_V^{s_1,p}(\mb R^N) \cap \widetilde{W}_V^{s_2,q}(\mb R^N)$ endowed with the norm
 \[ \|u\|:= \|u\|_{s_1,p}+ \|u\|_{s_2,q}.   \]
In order to deal with problem $(\mc P)$, we prove the following singular version of Moser-Trudinger type inequality for fractional Sobolev spaces in whole $\mb R^N$. For this we first obtain similar inequality for bounded domains much in the spirit of Adimurthi-Sandeep \cite[Theorem 2.1]{adisand}. Then, using Schwarz symmetrization technique we prove our theorem. For convenience, we denote
  \begin{align*}
  \Phi_\al(t) =e^{\al |t|^\frac{N}{N-s}} -\sum_{\substack{0\le j<N/s-1 \\ j\in\mb N} } \frac{\al^j}{j!} |t|^{j\frac{N}{N-s}}, \; \; \mbox{for }t\in\mb R.
   \end{align*}

We state as follows our first result.
 \begin{Theorem}\label{thm1}
 	Let $N\ge 2$, $s\in(0,1)$ and $p=N/s$. For all $\al>0$, $\ba\in[0,N)$ and $u\in W^{s,p}(\mb R^N)$, the following holds
 	 \begin{align*}
 	   \int_{\mb R^N} \frac{\Phi_\al(u)}{|x|^\ba}dx<\infty.
 	 \end{align*}
 	Furthermore, for all $\al< \big(1-\ba/N \big) \al_{N,s}$ and $\tau>0$,
 	  \begin{align*}
 	    \sup\bigg\{ \int_{\mb R^N} \frac{\Phi_\al(u)}{|x|^\ba}dx : u\in W^{s,p}(\mb R^N), \|u\|_{s,p,\tau}\le 1 \bigg\}<\infty,
 	    %\sup_{\substack { u\in W^{s,p}(\mb R^N) \\ \|u\|_{s,p,\tau}\le 1} } \int_{\mb R^N} \frac{\Phi_\al(u)}{|x|^\ba}dx<\infty,
 	  \end{align*}
 	  where $\| u\|_{s,p,\tau}= \left( [u]_{s,p}^p+\tau \int_{\mb R^N} |u|^p \right)^{1/p}$ and $\al_{N,s}>0$, is defined in Section \ref{prelm.} (see Theorem \ref{thm3}).
 \end{Theorem}
 The function $f$ is said to have subcritical growth condition with respect the exponential nonlinearity if it satisfies (f2) and the growth is critical if it satisfies (f2)$'$. Furthermore, we assume the following:
 \begin{enumerate}
 	\item[(f1)] The function $f:\mb R\to [0,\infty)$ is continuous with $f(t)=0$ for all $t\le 0$ and 
 	\[ \lim_{t\ra 0^+} \frac{f(t)}{t^{p-1}}=0. \]
 	\item[(f2)](Subcritical growth condition). For all $\al>0$, the following holds
 	\begin{align*}
 	\lim_{t\ra \infty} \frac{f(t)}{\Phi_\al(t)}=0.
 	\end{align*}
 	\item[(f3)] The map $t\mapsto t^{1-p}f(t)$ is nondecreasing in $(0, \infty)$ and $\ds\lim_{t\ra \infty} t^{-p}F(t)=\infty$, where $F(t)=\int_{0}^{t} f(\tau)d\tau$.
 \end{enumerate}
 For the critical growth condition, we assume $f$ satisfies the following conditions in addition to (f1). 
 \begin{enumerate}
 	\item[(f2)$'$](Critical growth). There exists $\al_0>0$ such that 
 	\begin{align*}
 	\lim_{t\ra \infty} \frac{f(t)}{\Phi_\al(t)}=0 \; \; \forall \al>\al_0 \; \mbox{ and} \quad 	\limsup_{t\ra \infty} \frac{f(t)}{\Phi_\al(t)}=\infty \; \; \forall \al<\al_0.
 	\end{align*}
 	\item[(f3)$'$] The map $t\mapsto t^{1-p}f(t)$ is nondecreasing in $(0,\infty)$ and there exists $\de>p$ such that $F(t)\ge C_\de t^\de$ for all $t\in\mb R^+$, for $C_\de>0$ sufficiently large (a lower bound is given in Lemma \ref{lem2}).
 	\item[(AR)] (Ambrosetti-Rabinowitz condition). There exists $\nu>p$ such that $\nu F(t)\le tf(t)$ for all $t\in\mb R^+$.
 \end{enumerate}
Due to the unbounded nature of the domain, Cerami sequences do not have the compactness property. We restore this compactness by exploiting the special  property of the potential $K$, namely \eqref{eqK} (see Lemma \ref{cmp}). This helps us to prove the strong convergence of Cerami sequences and hence to obtain nontrivial solutions. The existence of such sequences is obtained by using mountain pass lemma. In the subcritical case, we do not assume Ambrosetti-Rabinowitz type condition on $f$, which makes little difficult to prove boundedness of Cerami sequences. The non-homogeneous nature of the leading operator in $(\mc P)$ creates additional difficulty to prove boundedness of the Cerami sequence and its strong convergence result. Now, we state our main theorem as follows, which to the best of our knowledge, is new even for the case $\ba=0$. 
\begin{Theorem}\label{thm2}
	There exists a nonnegative nontrivial solution of problem $(\mc P)$ in the following cases
	\begin{enumerate}
	 \item[(i)] If (f1), (f2) and (f3) are satisfied.
	 \item[(ii)] If (f1), (f2)$'$, (f3)$'$ and (AR) are satisfied with $C_\de$, appearing in (f3)$'$, is sufficiently large.
	\end{enumerate}
\end{Theorem}
 \begin{Remark}
  We remark that the results of Theorems \ref{thm1} and \ref{thm2} are valid for equations of the type $(\mc P)$ involving a more general class of operators, for instance, operators of the form 
 	\begin{equation*}
 	\mc L_{\mc K_{r,s}} u(x)= 2\ds\lim_{\e\ra 0}\int_{\mb R^N\setminus B_\e(x)} |u(x)-u(y)|^{r-2}(u(y)-u(x))\mc K_{r,s}(x-y)dy,
 	\end{equation*}
  where $(r,s)\in\{(p,s_1),(q,s_2)\}$ with $1<q<p=N/s_1$ and $0<s_2\le s_1<1$. Here, the singular kernel $\mc K_{r,s}: \mb R^N\setminus\{0\}\to\mb R^+$ is such that 
 	\begin{enumerate}
 		\item [(i)] $m\mc K_{r,s}\in L^1(\mb R^N)$, where $m(x):=\min\{1,|x|^{r}\}$.
 		\item[(ii)] There exist $c_p>0$ and $c_q\ge 0$ such that $\mc K_{p,s_1}(x) \ge c_p |x|^{-(N+ps_1)}$ and $\mc K_{q,s_2}(x) \ge c_q |x|^{-(N+qs_2)}$  for $x\in \mb R^N\setminus\{0\}$.
 	\end{enumerate} 
  The corresponding energy space is defined as $X:= \widetilde{W}_V^{s_1,p}(\mb R^N) \cap \widetilde{W}_V^{s_2,q}(\mb R^N)$, where in the definition of  $\widetilde{W}^{s,r}_V(\mb R^N)$, the term  $|x-y|^{-(N+rs)} $ is replaced by $\mc K_{r,s}(x-y)$.\\
  For example, one can take $\mc K_{r,s}(x)= a_r(x) |x|^{-(N+rs)}$, where $a_r:\mb R^N\to\mb R$ are non-negative bounded functions, for $r\in\{p,q\}$, with $a_p$ is bounded away from zero.
 \end{Remark}
\noi\textbf{Notations:} For convenience in notation, we will use the following:
\begin{align*}
  &\mc A_1(u,v)=\int_{\mb R^{2N}}  \frac{|u(x)-u(y)|^{p-2}(u(x)-u(y))(v(x)-v(y))}{|x-y|^{N+ps_1}}~ dxdy,\;\text{ for all } u,v \in W^{s_1,p}(\mb R^N),
  \end{align*}
 and analogously $\mc A_2$ is defined in $W^{s_2,q}(\mb R^N)$.
 \begin{Definition}
    A function $u \in X$ is said to be  a solution of problem $(\mc P)$, if for all $v\in X$
 	\begin{align*}
 	\mc A_1(u,v)+\mc A_2(u,v)+ \int_{\mb R^N}V(x)\big(|u|^{p-2}+|u|^{q-2}\big)uv~dx
 	- \int_{\mb R^N} \frac{K(x) f(u)v}{|x|^\ba}dx=0.
 	\end{align*}
 \end{Definition}
 The Euler functional $\mc J: X \ra\mb R$ associated to the problem $(\mc P)$ is defined as
  \begin{align*}
  	\mc J(u)= \frac{1}{p} \|u\|_{s_1,p}^p +\int_{\mb R^N} V(x) |u|^p~dx+ \frac{1}{q} \|u\|_{s_2,q}^q+\int_{\mb R^N} V(x) |u|^q ~dx-\int_{\mb R^N} \frac{K(x)F(u(x))}{|x|^\ba}dx.
  \end{align*}

  \section{Some technical results}\label{prelm.}
  In this section, we first establish some  compact embedding results for space $X$.  We have the following notion of weighted Lebesgue space,
  \begin{align*}
  	L^p_V(\mb R^N):= \bigg\{ u:\mb R^N\to\mb R : \|u\|_{p,V}=\int_{\mb R^N} V(x)|u(x)|^p dx<\infty \bigg\},
  \end{align*}
  which is a Banach space when equipped with the norm $\| \cdot \|_{p,V}$, for $0<V\in C(\mb R^N)$.
 \begin{Remark}\label{rem1}
  \begin{enumerate}
  	 \item[(A)] Due to the fact $0\le \ba<N$, one can easily get that the embedding $X\hookrightarrow L^m(\mb R^N; |x|^{-\ba})$ is continuous for all $m\ge p$, that is, for all $m\ge p$ there exists $C_m>0$ such that for all $u\in X$,
  	  \[ \int_{\mb R^N} |u(x)|^m |x|^{-\ba}dx \le C_m \|u\|^m.\]
  	 \item[(B)] Let $\Om\subset\mb R^N$ be a bounded domain. Arguments similar to that of \cite[Remark 2.1]{doMySq} gives us $X$ is compactly embedded into $L^m(\Om)$. Indeed, by \cite[Theorem 7.1]{nezzaH} we have $W^{s_1,p}(\mb R^N)$ is compactly embedded into $L^p(\Om)$ and then using \eqref{emd} and interpolation argument, we can prove $X$ is compactly embedded  into $L^m(\Om)$ for all $m\ge p$. The aforementioned compact embedding result with a straight forward verification, yields $X$ is compactly embedded into $L^m(\Om; |x|^{-\ba})$ for all $m\ge p$.
  \end{enumerate}
 \end{Remark}
 \begin{Proposition} %\label{prop1}
  	The space $X$ is compactly embedded into $L^m_K(\mb R^N)$ for all $m\in(p,\infty)$.
  \end{Proposition}
  \begin{proof}
  The proof given here is an adaptation of the proof of \cite[Proposition 2.1]{miyPu} for $N=1$. Here we provide only sketch of the proof. For fixed $r>m>p$ and $\e>0$, there exists $\tau_0=\tau_0(\e)$, $\tau_1=\tau_1(\e)$ with $0<\tau_0<\tau_1$, $C=C(\e)>0$ and $C_0>0$ depending only on $K$, such that for all $x\in\mb R^N$ and $t\in\mb R$,
  		\begin{align}\label{eq21}
  		K(x)|t|^m \le \e C_0 \big( V(x)|t|^p+|t|^r \big)+ CK(x) \chi_{[\tau_0,\tau_1]}(|t|) |t|^m.
  		\end{align}
  Let $\{u_n \}\subset X$ be a bounded sequence, then by reflexive property of the space $X$, there exists $u\in X$ such that $u_n\rightharpoonup u$ weakly in $X$. From the continuous embedding of $X$ into $L^r(\mb R^N)$ and boundedness of the sequence $\{ \|u_n\| \}$, for some $M>0$, we have
  		\begin{align*}
  		\|u_n\|_{p,V}^p\le M \; \; \mbox{and}\quad \|u_n\|_\ga^\ga \le M \; \; \mbox{for all }n\in\mb N \; \mbox{ and }\ga\in \{m,r\}.
  		\end{align*}
  Therefore, $Q(u_n):= C_0\big( \|u_n\|_{p,V}^p+ \|u_n\|_r^r \big)\le 2C_0 M$ for all $n\in\mb N$. Set 
  		\[ A_\e^n:=\{x\in\mb R^N : \tau_0\le |u_n(x)|\le \tau_1 \}.   \]
  Then, by the fact that $\{u_n\}$ is bounded in $L^m(\mb R^N)$, it is easy to observe that $\{|A_\e^n|\}$ is bounded w.r.t. $n$. Again, by \eqref{eqK}, for $\e>0$, there exists $r_\e>0$ such that 
  		\begin{align*}
  		\int_{A_\e^n \cap B_{r_\e}^c(0)} K(x)dx <\frac{\e}{C \tau_1^m}, \mbox{ for all }n\in\mb N.
  		\end{align*}
  Using this together with the observation that $Q(u_n)$ is bounded, \eqref{eq21} gives us
  	 \begin{align}\label{eq22}
  		\int_{B_{r_\e}^c(0)} K(x)|u_n|^m \le 2C_0 M\e + C\tau_1^m \int_{A_\e^n \cap B_{r_\e}^c(0)} K(x)dx < (2C_0 M+1)\e,  \mbox{ for all }n\in\mb N.
  	 \end{align}
  Moreover, by compact embedding of the space $X$ into $L^\ga(B_{r_\e}(0))$ for all $\ga\ge p$ (see Remark \ref{rem1}), we get 
  	 \begin{align}\label{eq23}
  		 \lim_{n\ra\infty} \int_{B_{r_\e}(0)} K(x)|u_n|^m = \int_{B_{r_\e}(0)} K(x)|u|^m.
  	 \end{align}
  Therefore, \eqref{eq22} and \eqref{eq23}, implies
  	  \begin{align*}
  		 \lim_{n\ra\infty} \int_{\mb R^N} K(x)|u_n|^m = \int_{\mb R^N} K(x)|u|^m.
  	  \end{align*}
  From the above equation, it is easy to deduce that $u_n\ra u$ in $L^m_K(\mb R^N)$, as $n\ra\infty$.\QED
  \end{proof}
 We state the following  Moser-Trudinger type inequality for fractional Sobolev spaces in case of bounded domain.
\begin{Theorem}(\cite{parini,martin})\label{thm3}
 Let $\Om$ be a bounded domain in $\mb R^N$ $(N\ge 2)$ with Lipschitz boundary, and $s_1\in(0,1)$, $ps_1=N$. Let $\widetilde{W}_0^{s_1,p}(\Om)$ be the space defined as the completion of $C_c^\infty(\Om)$ with respect to $\| \cdot \|_{W^{s_1,p}(\mb R^N)}$ norm. Then, there exists $\al_{N,s_1}>0$ such that
	\begin{align*}
	 \sup\bigg\{ \int_{\Om} \exp\big( \al|u|^\frac{N}{N-s_1}\big) : u\in \tilde{W}_0^{s_1,p}(\Om), \|u\|_{p,s_1}\le 1 \bigg\}<\infty \quad \mbox{ for }\al\in [0,\al_{N,s_1}).
	\end{align*}
 Moreover,
	\begin{align*}
	\sup\bigg\{ \int_{\Om} \exp\big( \al|u|^\frac{N}{N-s_1}\big) : u\in \tilde{W}_0^{s_1,p}(\Om), \|u\|_{p,s_1}\le 1 \bigg\}=\infty \quad \mbox{ for }\al\in (\al^*_{N,s_1},\infty),
	\end{align*}
 where
	\[ \al^*_{N,s_1}= N\left( \frac{2(N\omega_N)^2\Ga(p+1)}{N!}\sum_{k=0}^{\infty}\frac{(N+k-1)!}{k!}\frac{1}{(N+2k)^p}  \right)^\frac{s_1}{N-s_1},     \]
	with $\omega_N$ as the volume of the $N$-dimensional unit ball.
\end{Theorem}
  Similar to the result of \cite{adisand}, we prove singular Moser-Trudinger inequality for fractional Sobolev spaces in bounded domains, which will help us to prove our Theorem \ref{thm1}.
  \begin{Lemma}\label{MTbd}
  	Let $\Om\subset \mb R^N$($N\ge 2$) be a bounded domain with Lipschitz boundary and let $u\in W^{s_1,p}_{0}(\Om)$. Then, for every $\al>0$ and $\ba\in[0,N)$,
  	\begin{align*}
  	\int_{\Om} \frac{e^{\al |u|^{N/(N-s_1)}}  }{|x|^\ba}dx <\infty.
  	\end{align*}
  	Moreover, if $\frac{\al}{\al_{N,s_1}}+\frac{\ba}{N}<1$, then
  	\begin{align}\label{eq15}
  	\sup_{\|u\|_{W^{s_1,p}(\Om)}\le 1}\int_{\Om} \frac{e^{\al |u|^{N/(N-s_1)}}  }{|x|^\ba}dx <\infty,
  	\end{align}
  \end{Lemma}
  \begin{proof}
  	 Let $t>1$ be such that $\ba t<N$, then using H\"older inequality and Theorem \ref{thm3}, we deduce that
  	\begin{align*}
  	  \int_{\Om} \frac{e^{\al |u|^{N/(N-s_1)}}  }{|x|^\ba}dx \le \left( \int_{\Om} e^{\al t' |u|^{N/(N-s_1)}}  dx \right)^\frac{1}{t'} \left(\int_{\Om} \frac{1}{|x|^{\ba t}}dx \right)^\frac{1}{t} <\infty.
  	\end{align*}
  	For the second part of the theorem, we first observe that there exist $\tilde{\al}\in(\al, \al_{N,s_1})$ and $t>1$ such that $\frac{\al}{\tilde{\al}}+\frac{\ba t}{N}=1$ (this can be done by first choosing $\tilde{\al}<\al_{N,s_1}$ such that $\frac{\al}{\al_{N,s_1}}+\frac{\ba}{N}<\frac{\al}{\tilde{\al}}+\frac{\ba}{N}<1$). Now by H\"older inequality, we have
  	\begin{align*}
  	  \sup_{[u]_{s_1,p}\le 1}\int_{\Om} \frac{e^{\al |u|^{N/(N-s_1)}}  }{|x|^\ba}dx \le \sup_{[u]_{s_1,p}\le 1} \left(  \int_{\Om} e^{\tilde\al |u|^{N/(N-s_1)}}  dx   \right)^\frac{\al}{\tilde\al} \left(\int_{\Om} \frac{1}{|x|^{N/t}} \right)^{\ba t/N},
  	\end{align*}
  	since $\tilde{\al}<\al_{N,s_1}$ and $t>1$, by Theorem \ref{thm3}, we get that the the right hand side quantity is finite. \QED
  \end{proof}
  Before proving Theorem \ref{thm1}, we state the following radial lemma.
  \begin{Lemma}\label{rad}
  	Let $N\ge 2$ and $u\in L^p(\mb R^N)$, with $1\le p<\infty$, be a radially symmetric non-increasing function. Then
  	\begin{align*}
  		|u(x)|\le |x|^{-N/p} \Big(\frac{N}{\omega_{N-1}} \Big)^{1/p} \|u\|_p, \;\mbox{for }x\neq 0,
  	\end{align*}
  	where $\omega_{N-1}$ is the $(N-1)$ dimensional measure of $(N-1)$ sphere.
  \end{Lemma}
  
 \subsection{Proof of Theorem \ref{thm1}} 
 Without loss of generality, we assume $u\ge 0$ and let $u^*$ be the Schwarz symmetrization of $u$.  Then by (\cite{almgren,brock}), for any continuous and increasing function $G:[0,\infty)\to [0,\infty)$, there holds
 \begin{align*}
 \int_{\mb R^N} G(u^*(x))dx=\int_{\mb R^N} G(u(x))dx.
 \end{align*}
 Moreover, for all $u\in W^{s,p}(\mb R^N)$ and $1\le m<\infty$, $u^*\in W^{s,p}(\mb R^N)$ with
 \begin{align}\label{eq24}
 	 \int_{\mb R^N}\int_{\mb R^N} \frac{|u^*(x)-u^*(y)|^p}{|x-y|^{2N}}dxdy \le \int_{\mb R^N}\int_{\mb R^N} \frac{|u(x)-u(y)|^p}{|x-y|^{2N}}dxdy \; \; \mbox{ and } \|u^*\|_m=\|u\|_m.
 \end{align}
 Therefore, by Hardy-Littlewood inequality for symmetrization and the fact that $\big(1/|x|^\ba \big)^*=1/|x|^\ba$, we get
 \begin{align}\label{eqS}
 	\int_{\mb R^N} \frac{\Phi_\al(u)}{|x|^\ba}~dx\le 	\int_{\mb R^N} \frac{\Phi_\al(u^*)}{|x|^\ba}~dx.
 \end{align}
  Fix $R>0$ (to be specified later), we have
  \begin{align}\label{eq25}
  	\int_{|x|>R} \frac{\Phi_\al(u^*)}{|x|^\ba}~dx = \int_{|x|>R}  \frac{1}{|x|^\ba} \sum_{j=k_0}^{\infty} \frac{\al^j}{j!} |u^*|^{jp'}~dx,
  \end{align}
 where $k_0$ is the smallest integer such that $k_0\ge p-1$ and $p'=p/(p-1)$ is the H\"older conjugate of $p$. Now we consider the following cases:\\
 \textit{Case (i)}: For all $j\ge k_0>p-1$.\\
  Using Lemma \ref{rad} and \eqref{eq24}, we obtain
 \begin{align}\label{eq26}
 	 \int_{|x|>R} \frac{|u^*|^{jp'}}{|x|^\ba} &\le \int_{|x|>R} |x|^{-\frac{N}{p-1}j-\ba} \Big(\frac{N}{\omega_{N-1}} \Big)^\frac{j}{p-1} \|u^*\|_p^{jp'} \nonumber \\
 	 &\leq \Big(\frac{N}{\omega_{N-1}} \Big)^\frac{j}{p-1} \|u\|_p^{jp'}  R^{N-\frac{N}{p-1}j-\ba}.
 \end{align}
 \textit{Case (ii)}: If $k_0=p-1$.\\
  Using \eqref{eq24}, we obtain
  \begin{align}\label{eq27}
     \int_{|x|>R} \frac{|u^*|^{k_0p'}}{|x|^\ba} \le \int_{|x|>R} \frac{|u^*(x)|^p}{|x|^\ba}\le \frac{1}{R^\ba}\int_{\mb R^N} |u^*(x)|^p dx=\frac{\|u\|_p^p}{R^\ba}.
 \end{align}
 Then, coupling \eqref{eq26} and \eqref{eq27} in \eqref{eq25}, we get
 \begin{align*}
 	 \int_{|x|>R} \frac{\Phi_\al(u^*)}{|x|^\ba} \le \frac{1}{R^\ba}  \Big( C_\al \|u\|_p^p + \sum_{j=k_0+1}^{\infty} \frac{\al^j}{j!}  \Big(\frac{N}{\omega_{N-1}} \Big)^\frac{j}{p-1} \|u\|_p^{jp'}  R^{N-\frac{N}{p-1}j} \Big).
 \end{align*}
 For fixed $u\in W^{s,p}(\mb R^N)$, we choose $R>0$ such that
    \[ R^{-\frac{N}{p-1}} \Big(\frac{N}{\omega_{N-1}} \Big)^{1/(p-1)} \|u\|_p^{p'} =1,\]
this implies
\begin{align}\label{eq28}
   \int_{|x|>R} \frac{\Phi_\al(u^*)}{|x|^\ba} \le \frac{1}{R^\ba} C(N,s,\al,\|u\|_p)<\infty.
\end{align}
Next, for $p'=N/(N-s)$, there exists $B=B(N,s)>0$ such that for all $\e>0$,
\begin{align}\label{eq29}
	(u+v)^{p'}\le u^{p'}+ B u^{p'-1} v+v^{p'}, \quad \mbox{and }u^\ga v^\ga \le \e u+ \e^{-\ga/\ga'} v
\end{align}
for all $u,v\ge 0$ and  $\ga,\ga'>0$ satisfying $\ga+\ga'=1$. For fixed $x_0\in\mb R^N$ with $|x_0|=1$, define
\begin{align*}
	v(x)=\begin{cases}
	   u^*(x)-u^*(Rx_0),  \; \mbox{if }x\in B_R(0) \\
	   0, \qquad \mbox{if }x\in\mb R^N\setminus B_R(0).
	\end{cases}
\end{align*}
Then, since $u^*$ is radially decreasing function, we have $v\ge 0$ and by \cite[Lemma 2.2]{xiang},
 \begin{align*}
 	[v]_{s,p}^p \le [u^*]_{s,p}^p \le [u]_{s,p}^p<\infty.
 \end{align*}
 Therefore, $v\in W^{s,p}(\mb R^N)$ with $v=0$ a.e. in $\mb R^N\setminus B_R(0)$. Using \eqref{eq29}, for $x\in B_R(0)$, we deduce that
  \begin{align*}
  	|u^*(x)|^{p'}= |v+u^*(Rx_0)|^{p'}\le v^{p'}+ B v^{p'-1} u^*(Rx_0)+u^*(Rx_0)^{p'},
  \end{align*}
 and
 \[ v^{p'-1}u^*(Rx_0)= (v^{p'})^{(p'-1)/p'} (u^*(Rx_0)^{p'})^{1/p'} \le \frac{\e}{A} v^{p'}+ \big(\frac{\e}{A}\big)^{-1/(p'-1)} u^*(Rx_0)^{p'}. \]
 Thus,
 \begin{align*}
 |u^*(x)|^{p'} \le (1+\e) v^{p'}+ C(\e,s,N) u^*(Rx_0)^{p'},
 \end{align*}
 where $C(\e,s,N)= 1+\big(A/\e\big)^{1/(p'-1)}$. Therefore, using Lemma \ref{MTbd}, we obtain %using this observation, we obtain
 \begin{align}\label{eq30}
 		\int_{|x|\le R} \frac{\Phi_\al(u^*)}{|x|^\ba} \le \int_{|x|\le R} \frac{e^{\al |u^*|^{p'}}}{|x|^\ba} \le  e^{\al C(\e,s,N) |u^*(Rx_0)|^{p'}} \int_{|x|\le R} \frac{e^{\al (1+\e) |v|^{p'}}}{|x|^\ba}<\infty.
 \end{align}
 This together with \eqref{eq28} and \eqref{eqS} proves the first part of the Theorem.\\
 For the second part, we consider $u\in W^{s,p}(\mb R^N)$ such that $\|u\|_{s,p,\tau}\le 1$.
 From \eqref{eq26}, we have
 \begin{equation}\label{eq00}
  \begin{aligned}
 \int_{|x|>R} \frac{|u^*|^{jp'}}{|x|^\ba} &\le  \Big(\frac{N}{\omega_{N-1}} \Big)^\frac{j}{p-1} \|u\|_p^{jp'}  R^{N-\frac{N}{p-1}j-\ba} \\
 &\le R^{N-\ba} \Big(\frac{N}{\omega_{N-1}} \Big)^\frac{j}{p-1} \tau^{-\frac{j}{p-1}} R^{-\frac{N}{p-1}j},
 \end{aligned}
 \end{equation}
 where in the last inequality we used the fact $\|u\|_{s,p,\tau}\le 1$. Choosing $R>0$ such that
 \[ R^{-N} \frac{N}{\omega_{N-1}\tau}=1,  \]
 then, on account of \eqref{eq00}, we deduce from \eqref{eq25} that 
 \begin{align}\label{eq31}
 \int_{|x|>R} \frac{\Phi_\al(u^*)}{|x|^\ba} \le R^{N-\ba} \sum_{j=k_0}^{\infty} \frac{\al^j}{j!}\le C(N,s,\al,\ba,\tau).
 \end{align}
Now due to the fact that $\|u\|_{s,p,\tau}\le 1$ and $v(x)\le u^*(x)$ in $B_R(0)$, we have
 \begin{align*}
 	\|v\|_{s,p,\tau}^p=[v]_{s,p}^p+\tau \|v\|_p^p\le [u^*]_{s,p}^p+\tau \|u^*\|_p^p \leq [u]_{s,p}^p+\tau \|u\|_p^p\le 1,
 \end{align*}
 and by using the radial lemma \ref{rad}, we obtain
 \begin{align*}
 	u^*(Rx_0)^{p'}\le |Rx_0|^{-N/(p-1)} \Big(\frac{N}{\omega_{N-1}} \Big)^{1/(p-1)} \|u^*\|_p^{p'} \le R^{-N/(p-1)} \Big(\frac{N}{\omega_{N-1}\tau} \Big)^{1/(p-1)}.
 \end{align*}
 Therefore, from \eqref{eq30} and \eqref{eq15}, we get
 \begin{align}\label{eq32}
    \int_{|x|\le R} \frac{\Phi_\al(u^*)}{|x|^\ba} \le e^{C(\e,s,N,\tau,\al,\ba)} \int_{|x|\le R} \frac{1}{|x|^\ba} \exp\{\al (1+\e) \|v\|_{s,p,\tau}^{p'} \big|\frac{v}{\|v\|_{s,p,\tau}}\big|^{p'}\} \le C(N,s,\tau,\al,\ba)
 \end{align}
if we choose $\e>0$ such that $\al(1+\e)<\big(1-\ba/N\big) \al_{N,s}$.
 Taking into account \eqref{eq31}, \eqref{eq32} and \eqref{eqS}, we complete the proof of the second part of the Theorem. \QED
 Next, we establish the compactness result under the assumption that (f1) through (f3) hold, that is, the subcritical case.
 %It is worth noticing that compactness results do not hold to the functional $\mc J$ for general nonlinearity, e.g., when $K(x)\equiv 1$,  due to unbounded nature of the domain $\mb R^N$.
 \begin{Lemma}\label{cmp}
 	Let $\{u_n\}\subset X$ be a sequence such that $u_n\rightharpoonup u$ weakly in $X$, for some $u\in X$. Then up to a subsequence, the following hold
 	\begin{align*}
 	&\lim_{n\ra\infty} \int_{\mb R^N} K(x)|x|^{-\ba}F(u_n(x))dx = \int_{\mb R^N} K(x)|x|^{-\ba}F(u(x))dx \\
 	&\lim_{n\ra\infty} \int_{\mb R^N} {K(x)|x|^{-\ba}f(u_n(x))u_n(x)}dx = \int_{\mb R^N} {K(x)|x|^{-\ba}f(u(x))u(x)}dx \\
 	&\lim_{n\ra\infty} \int_{\mb R^N} {K(x)|x|^{-\ba}f(u_n(x))v(x)}dx = \int_{\mb R^N} {K(x)|x|^{-\ba}f(u(x))v(x)}dx, \; \; \mbox{for all }v\in X.
 	\end{align*}
 	\end{Lemma}
\begin{proof}
	Set $M=\sup_n \|u_n\|$. For $0<\al <\big(1-\frac{\ba}{N}\big)\frac{\al_{N,s_1}}{M^{N/(N-s_1)}}$, from (f1) and (f2), we get 
	\begin{align*}
	\limsup_{t\ra \infty} \frac{f(t)t}{\Phi_\al(t)}= \limsup_{t\ra \infty} \frac{F(t)}{\Phi_\al(t)} =0, \; \mbox{and } \limsup_{t\ra 0} \frac{f(t)t}{|t|^p}=\limsup_{t\ra 0} \frac{F(t)}{|t|^p}=0.
	\end{align*}
	Therefore, for $\e>0$ and $\de>p$, there exist $\rho_0=\rho_0(\e)$, $\rho_1=\rho_1(\e)$ with $0<\rho_0<\rho_1$, $C=C_\e>0$ and $C_0>0$ depending only on $K$, such that for all $x\in\mb R^N$ and $t\in\mb R$, the following hold
	\begin{align}\label{eqKF}
	&|K(x)F(t)| \le \e C_0 \big(|t|^p+\Phi_\al(t) \big)+ CK(x) \chi_{[\rho_0,\rho_1]}(|t|) |t|^\de  \nonumber \\
	&|K(x)f(t)t| \le \e C_0 \big(|t|^p+\Phi_\al(t) \big)+ CK(x) \chi_{[\rho_0,\rho_1]}(|t|) |t|^\de.
	\end{align}
	By the embedding results of $X$ into $L^m(\mb R^N)$ (and hence into $L^m(\mb R^N; |x|^{-\ba})$, for $0\le\ba<N$), we have 
	\begin{align}\label{eq33}
	\sup_n\int_{\mb R^N}\frac{|u_n|^p}{|x|^\ba}\le \tilde{M},
	\end{align}
	for some $\tilde{M}\ge M>0$. Now, for $\al<\big(1-\frac{\ba}{N}\big)\frac{\al_{N,s_1}}{M^{N/(N-s_1)}}$, we have 
	\begin{align*}
	\al \|u_n\|^{N/(N-s_1)}\le \al M^{N/(N-s_1)}< \big(1-\ba/N)\al_{N,s_1},
	\end{align*}
	therefore, by Theorem \ref{thm1} and the fact that $\Phi_\al$ is increasing with respect to $\al$, we obtain 
	\begin{align}\label{eq34}
	\sup_n\int_{\mb R^N} \frac{\Phi_\al(u_n)}{|x|^\ba} \le \sup_n \int_{\mb R^N} \frac{\Phi_{\al M^{N/(N-s_1)}}(\frac{u_n}{\|u_n\|})}{|x|^\ba}\le \tilde{M}.
	\end{align} 
	Let 	$ A_\e^n:=\{x\in\mb R^N : \rho_0\le |u_n(x)|\le \rho_1 \}$.
	Then, as $\{|A_\e^n| \}$ is bounded w.r.t. $n$, using \eqref{eqK}, we deduce that 
	\begin{align*}
	\lim_{r\ra\infty} \Big| \int_{A_\e^n \cap B_{r}^c(0)} \frac{K(x)}{|x|^\ba}dx \Big|\le \lim_{r\ra\infty} \frac{1}{r^\ba} \Big|\int_{A_\e^n \cap B_{r}^c(0)} K(x)dx\Big| =0, \mbox{ uniformly w.r.t. }n\in\mb N.
	\end{align*}
	Therefore, for $\e>0$, there exits $R_\e>0$ such that 
	\begin{align}\label{eq35}
	\int_{A_\e^n \cap B_{R_\e}^c(0)} \frac{K(x)}{|x|^\ba}dx <\frac{\e}{C \rho_1^\de} \; \; \mbox{for all }n\in\mb N.
	\end{align}
	Taking into account \eqref{eqKF} through \eqref{eq35}, we obtain
	\begin{align}\label{eq36}
	&\int_{B_{R_\e}^c(0)} \frac{K(x)F(u_n(x))}{|x|^\ba}dx \le 2C_0 \tilde{M}\e + C\rho_1^\de \int_{A_\e^n \cap B_{R_\e}^c(0)} \frac{K(x)}{|x|^\ba}dx < (2C_0 \tilde{M}+1)\e, \nonumber \\
	&\int_{B_{R_\e}^c(0)} \frac{K(x)f(u_n(x))u_n(x)}{|x|^\ba}dx \le 2C_0 \tilde{M}\e + C\rho_1^\de \int_{A_\e^n \cap B_{R_\e}^c(0)} \frac{K(x)}{|x|^\ba}dx < (2C_0 \tilde{M}+1)\e,
	\end{align}
	for all $n\in\mb N$. 
	Furthermore, from (f1) and (f2), it is easy to observe that 
	\[ |f(t)|\le C_1\big( |t|^p+\Phi_\al(t) \big), \mbox{ for all }t\in\mb R,  \]
	where $C_1>0$ is a constant. Therefore, using the fact that $K\in L^\infty(\mb R^N)$, we get
	\begin{align*}
	\Bigg| \int_{B_{R_\e}(0)} \frac{K(x)f(u_n)(u_n-u)}{|x|^\ba}dx \Bigg| \le C \left( \int_{B_{R_\e}(0)} \frac{|u_n|^p|u_n-u|}{|x|^\ba} +\int_{B_{R_\e}(0)} \frac{\Phi_\al(u_n)|u_n-u|}{|x|^\ba}\right).
	\end{align*}
	We choose $\ga>1$ close to $1$ such that $\ga'>p$ and $\ga\al < \big(1-\frac{\ba}{N}\big)\frac{\al_{N,s_1}}{M^{N/(N-s_1)}}$. Then, using H\"older inequality, Theorem \ref{thm1} and the fact that $\{ \|u_n\| \}$ is bounded, we deduce that 
	\begin{align*}
	\Bigg| \int_{B_{R_\e}(0)} \frac{K(x)f(u_n)(u_n-u)}{|x|^\ba}dx \Bigg| \le &C \left[\left( \int_{\mb R^N} \frac{|u_n|^{\ga p} }{|x|^\ba} \right)^{1/\ga} +\left( \int_{\mb R^N} \frac{\Phi_\al(u_n)^\ga}{|x|^\ba}\right)^{1/\ga} \right] \\
	&\quad \times\left( \int_{B_{R_\e}(0)} \frac{|u_n-u|^{\ga'} }{|x|^\ba} \right)^{1/\ga'} \\
	& \le C \left( \int_{B_{R_\e}(0)} \frac{|u_n-u|^{\ga'} }{|x|^\ba} \right)^{1/\ga'}  \ra 0 \; \mbox{ as }n\ra\infty,
	\end{align*}
	where in the last line we have used the compact embedding result of  $X$ given in Remark \ref{rem1}. Hence, 
	\begin{align*}
	\lim_{n\ra\infty}\int_{B_{R_\e}(0)} \frac{K(x)f(u_n)u_n}{|x|^\ba}dx =\int_{B_{R_\e}(0)} \frac{K(x)f(u_n)u_n}{|x|^\ba}dx.
	\end{align*}
	Using (f3), one can easily verify that $ pF(t)\le f(t)t$ for all $t\in\mb R^N$. Therefore, by generalized Lebesgue dominated convergence theorem, similar convergence result holds for $F$ also. Thus, using \eqref{eq36}, we get the required convergence result of the first two integrals of the Lemma.\\ 
	% 	 Furthermore, from \eqref{eqKF}, it is clear that 
	% 	 \begin{align*}
	% 	 		|K(x)F(t)| \le  C_0 \big(|t|^r+\Phi_\al(t) \big)+ CK(x) |t|^\de  \mbox{ and }  
	% 	 	|K(x)f(t)t| \le  C_0 \big(|t|^r+\Phi_\al(t) \big)+ CK(x) |t|^\de.
	% 	 \end{align*}
	% 	Now, by  the embedding of $X$ into $L^m(\mb R^N;|x|^\ba)$, and using \eqref{eq33} and \eqref{eq34}, we get 
	% 	\begin{align*}
	% 		 \int_{B_{R_\e}(0)} \frac{K(x)F(u_n(x))}{|x|^\ba}dx, \; \int_{B_{R_\e}(0)} \frac{K(x)f(u_n)u_n}{|x|^\ba}dx \le C_2 \mbox{ for all }n\in\mb N,
	% 	\end{align*}
	% 	where $C_2>0$ is a constant and due to compact embeddings, we have 
	% 	 \[  u_n\ra u \mbox{ in } L^m(B_{R_\e}(0)), \forall m\ge N. \]
	% 	Similar procedure yields, $\frac{K(x)f(u(x))}{|x|^\ba}\in L^m\big(B_{R_\e}(0)\big)$ for all $m\ge N$. Then, by \cite[Lemma 2.1]{fig}, we obtain
	% 	\begin{align*}
	% 		&\lim_{n\ra\infty} \int_{B_{R_\e}(0)} \frac{K(x)F(u_n(x))}{|x|^\ba}dx=\int_{B_{R_\e}(0)} \frac{K(x)F(u(x))}{|x|^\ba}dx \;\mbox{ and } \\
	% 		&\lim_{n\ra\infty}\int_{B_{R_\e}(0)} \frac{K(x)f(u_n)u_n}{|x|^\ba}dx =\int_{B_{R_\e}(0)} \frac{K(x)f(u_n)u_n}{|x|^\ba}dx.
	% 	\end{align*}
	Next, to prove the last convergence result of the Lemma, we set $E_n:=\{ x\in\mb R^N: |u_n(x)|\le 1\}$ and $E:=\{ x\in\mb R^N: |u(x)|\le 1\}$. We first claim that 
	the sequence $\{K(x)f(u_n)\chi_{E_n} \}$ is uniformly bounded in $L^{r'}(\mb R^N; |x|^{-\ba})$, where $r'$ is the H\"older conjugate of $r$. By (f1), it is easy to see that $|f(t)|\le C|t|^{p-1}$ for all $|t|\le 1$ and some $C>0$. Therefore,
	\[ |f(u_n)|\le C|u_n|^{p-1} \mbox{ in } E_n, \mbox{ for all }n\in\mb N.  \]
	Using the fact that $X\hookrightarrow L^p(\mb R^N; |x|^{-\ba})$ is continuous and $\{\|u_n\| \}$ is bounded, we obtain
	\begin{align*}
	\int_{E_n} \frac{|K(x)f(u_n)|^{p'}}{|x|^\ba} \le C\int_{E_n} \frac{|u_n|^p}{|x|^\ba}\le C\|u_n\|^p\le C, \mbox{ for all }n\in\mb N.
	\end{align*}
	This together with the pointwise convergence gives us
	\begin{align*}
	\lim_{n\ra\infty}\int_{E_n} \frac{K(x)f(u_n)\phi}{|x|^\ba}dx =\int_{E} \frac{K(x)f(u_n)\phi}{|x|^\ba}dx \;\; \mbox{ for all }\phi\in L^p(\mb R^N;|x|^{-\ba}).
	\end{align*}
	Now, for any $v\in X$, we have $v\in L^r(\mb R^N;|x|^{-\ba})$ and hence
	\begin{align*}
	\lim_{n\ra\infty}\int_{E_n} \frac{K(x)f(u_n)v}{|x|^\ba}dx =\int_{E} \frac{K(x)f(u_n)v}{|x|^\ba}dx.
	\end{align*}
	Similarly, by (f2), for $m\ge 1$, we obtain 
	\begin{align*}
	|f(u_n(x))|^m\le C  \Phi_\al(u_n(x))^m\le C \Phi_{m\al}(u_n(x)) \; \; \mbox{ for }x\in E_n^c \ \mbox{ and for all }n\in\mb N. 
	\end{align*}
	We choose $m>1$ close to $1$ such that $m'>p$ and $m\al<\big(1-\frac{\ba}{N}\big)\frac{\al_{N,s_1}}{M^{N/(N-s_1)}}$. Then, by Theorem \ref{thm1}, we get 
	\[ \int_{E_n^c} \frac{|K(x)f(u_n)|^m}{|x|^\ba} \; \mbox{ is uniformly  bounded}.   \]
	Therefore, for $v\in X$, we have $v\in L^{m'}(\mb R^N;|x|^{-\ba})$ and pointwise convergence yields
	\begin{align*}
	\lim_{n\ra\infty}\int_{E_n^c} \frac{K(x)f(u_n)v}{|x|^\ba}dx =\int_{E^c} \frac{K(x)f(u_n)v}{|x|^\ba}dx.
	\end{align*}
	This completes proof of the lemma. \QED
\end{proof}
Without loss of generality, we may assume $\al_0=\big(1-{\ba/N}\big)\al_{N,s_1}$, appearing in (f2)$'$. Then, we have similar compactness result if the conditions (f1), (f2)$'$ and (f3)$'$ hold, that is, the critical case.
 \begin{Corollary}\label{cor1}
 	Let $\{v_n\}\subset X$ be a sequence such that $v_n\rightharpoonup v$ weakly in $X$, for some $v\in X$ and
 	\[ L:=\sup_n \|v_n\|\in(0,1). \]
 	Then, the convergence results of the Lemma \ref{cmp} are true in this case also.
 \end{Corollary}
 \begin{proof}
 	Since $L\in(0,1)$, there exists $\al_L> \al_{N,s_1}\big(1-\frac{\ba}{N}\big)$ such that $\al_L< \big(1-\frac{\ba}{N}\big)\frac{\al_{N,s_1}}{L^{N/(N-s_1)}}$. Then, by (f1) and (f2)$'$, results similar to \eqref{eqKF} hold in this case too, with $\al$ replaced by $\al_L$. Furthermore, Theorem \ref{thm1} can be applied to obtain boundedness results as in \eqref{eq33} and \eqref{eq34}. Now, rest of the proof follows similar to that of the Lemma with $\al$ replaced by $\al_L$. \QED
 \end{proof}
 It is easy to verify that the functional $\mc J$ is of class $C^1(X)$. Now, we verify the mountain pass geometry for $\mc J$.
 \begin{Lemma}
 	The functional $\mc J$ satisfies the following
 	\begin{enumerate}
 	 \item[(I)] there exists $v_0\in X\setminus\{0\}$ with $\|v_0\| \ge 2$ such that $\mc J(v_0)<0$.
 	 \item[(II)] There exists $\eta>0$ and $\rho\in(0,1)$ such that $\mc J(v)\ge \eta$ for all $v\in X$ with $\|v\|=\rho$.
 	\end{enumerate}
 \end{Lemma}
 \begin{proof}
 	(I) Proof of this part is a standard procedure and follows by the super linear nature of the nonlinearity $F$ with respect to $p$. \\
 	(II) For fixed $\rho_0\in(0,1)$, we choose $\al>0$ such that $ \al_{N,s_1}\big(1-\frac{\ba}{N}\big) <\al <\big(1-\frac{\ba}{N}\big)\frac{\al_{N,s_1}}{\rho_0^{N/(N-s_1)}}$. Now, by the fact that $K\in L^\infty(\mb R^N)$, (f1) and (f2) (or (f2)$'$), for $\de>p$, we have
 	\begin{align}\label{eq37}
 		K(x) F(t)\le \frac{1}{2^p p C_p^p} t^p + C_2 \Phi_\al(t) t^\de \; \mbox{ for all }t\in\mb R^+ \; \mbox{and }x\in\mb R^N,
 	\end{align}
 	where $C_2>0$ is a constant and $C_p$ appears in Remark \ref{rem1} (A). We choose $m>1$ close to $1$ such that $m'>p$ and $m\al< \big(1-\frac{\ba}{N}\big)\frac{\al_{N,s_1}}{\rho_0^{N/(N-s_1)}}$, then by Theorem \ref{thm1} and the embedding of $X$ into $L^\ga(\mb R^N;|x|^{-\ba})$, for $\ga\ge p$, we obtain 
 	\begin{align}\label{eq38}
 	  \int_{\mb R^N} \frac{\Phi_\al(w)|w|^\de}{|x|^\ba} \le \left(\int_{\mb R^N} \frac{|\Phi_\al(w)|^m}{|x|^\ba} \right)^{1/m} \left( \int_{\mb R^N} \frac{|w|^{\de m'}}{|x|^\ba} \right)^{1/m'} &\le C_3 \left(\int_{\mb R^N} \frac{\Phi_{m\al}(w)}{|x|^\ba} \right)^{1/r} \|w\|^\de \nonumber \\
 	  &\le C_4 \|w\|^\de,
 	\end{align}
  	for all $w\in X$ with $\|w\|=\rho\le \rho_0$, where $C_3,C_4>0$ are constants independent of $w$. Therefore, using \eqref{eq37}, \eqref{eq38} and Remark \ref{rem1} (A), we deduce that
 	\begin{align*}
 	  \mc J(w) &\ge \frac{1}{p} \|w\|_{s_1,p}^p+ \frac{1}{q} \|w\|_{s_2,q}^q- \frac{1}{p2^p C_p^p} \int_{\mb R^N} \frac{|w|^{p}}{|x|^\ba} - C_2 \int_{\mb R^N} \frac{\Phi_\al(w)|w|^\de}{|x|^\ba} \\
 	  &\ge \frac{2^{1-p}}{p} \|w\|^p -\frac{1}{p2^p C_p^p} C_p^p  \|w\|^p -C_4 \|w\|^\de= \frac{2^{-p}}{p} \|w\|^p-C_4 \|w\|^\de,
 	\end{align*}
 	where we have used the fact that $\|w\|_{s_1,p}, \|w\|_{s_2,q} \le \|w\| <1$ and $C_i$'s are positive constants. By the fact that $\de>p$, there exists $\eta>0$ and $\rho$ small enough such that $\mc J(w)\ge \eta$ for all $w\in X$ with $\|w\|=\rho$. \QED
 \end{proof}
 The mountain pass lemma ensures the existence of a Cerami sequence at the mountain pass level, that is, there exists a sequence $\{u_n\}\subset X$ such that
 \begin{align*}
 	\mc J(u_n)\ra c \; \; \mbox{ and } (1+\|u_n\|) \|\mc J^\prime(u_n) \| \ra 0, \; \; \mbox{as }n\ra\infty,
 \end{align*}
 where $c:= \ds\inf_{g\in\Ga} \max_{t\in[0,1]} \mc J(g(t))$ with $\Ga= \{ g\in C([0,1],X) : g(0)=0 \; \mbox{and }\mc J(g(1))<0 \}$.
 \begin{Lemma}\label{non}
 	Every solution $u$ of $(\mc P)$ is nonnegative and if $\{u_n \}\subset X$ is a Cerami sequence, then $\|u_n^-\| \ra 0$, as $n\ra\infty$.
 \end{Lemma}
\begin{proof}
  Proof follows using the inequalities,
	\begin{align*}
		(u(x)-u(y)) (u^-(x)-u^-(y)) \le -|u^-(x)-u^-(y)|^2 \; \mbox{ and } |u(x)-u(y)| \ge |u^-(x)-u^-(y)|.
	\end{align*}
  Applying these one can deduce that $\mc A_2(u,u^-)+\int_{\mb R^N} V(x)|u|^{q-2}u u^-\le 0$. Then, rest of the proof follows similar to \cite[Lemma 2.9]{miyPu}.\QED
\end{proof}
Following the standard procedure, we can prove the following result.
 \begin{Lemma}
 	Suppose the function $f$ satisfies (f1), (f3)$'$ and (AR). Then, any Cerami sequence of $\mc J$ at level $c$ is bounded.
 \end{Lemma}
\begin{Lemma}\label{lem2}
	Let the function $f$ satisfies (f1), (f2)$'$, (f3)$'$ and (AR). Then, for any Cerami sequence $\{u_n\}\subset X$ for $\mc J$ at the mountain pass level $c$, the following holds
	\[ \sup_{n\in\mb N} \|u_n\|\in(0,1),  \]
	if the constant $C_\de$, appearing in (f3)$'$, is sufficiently large.
\end{Lemma}
\begin{proof}
	Fix $\psi\in C_c^\infty(\mb R^N)$ with $\| \psi \|>0$. For $\de>p$, set $K_0 := \ds\inf_{supp(\psi)} K>0$ and $S_\de =\frac{\| \psi \|}{\| \psi\|_\de}>0$. Now, using (f3)$'$, for $l>1$, we have 
	\begin{align*}
	\mc J(l\psi) = \frac{l^p}{p} \|\psi\|_{s_1,p}^p+ \frac{l^q}{q} \|\psi\|_{s_2,q}^q-  \int_{\mb R^N} \frac{K(x) F(l\psi(x))}{|x|^\ba}  \le \frac{l^p}{q} \|\psi\|^p- K_0 C_\de l^\de S_\de^{-\de} \| \psi \|^\de.
	\end{align*}
	Since $\de>p$, there exits $l_\de>0$ sufficiently large such that $\mc J(l_\de \psi)<0$. Therefore, 
	%$g(t)=t l_\de \psi$, for $t\in[0,1]$, is of class $C([0,1],X)$ and belongs to $\Ga$. Hence,
	\begin{align}\label{eq39}
	c= \ds\inf_{\ga\in\Ga} \max_{t\in[0,1]} \mc J(\ga(t)) \le \max_{t\in[0,1]} \mc J(t l_\de \psi)\le \sup_{t\in\mb R^+} \mc J(t\psi). 
	\end{align}
	Consider $h: [0,\infty)\to \mb R$ defined by $h(t):= \frac{t^p}{p} \|\psi\|^p+ \frac{t^q}{q} \|\psi\|^q- C(\de) t^\de \| \psi \|^\de$, where $C(\de)=K_0 C_\de  S_\de^{-\de}$. Then, an easy computation yields
	\begin{align*}
	\sup_{t\ge 0} h(t) &\le \sup_{t\ge 0} \left(\frac{t^p}{p} \|\psi\|^p- \frac{1}{2}C(\de) t^\de \| \psi \|^\de \right) +\sup_{t\ge 0} \left( \frac{t^q}{q} \|\psi\|^q- \frac{1}{2}C(\de) t^\de \| \psi \|^\de\right) \\
	&= \left(\frac{2}{C(\de)}\right)^\frac{p}{\de-p} \left(\frac{1}{p}-\frac{1}{\de}\right) \big( \|\psi\|^{p-\de} \big)^\frac{\de}{\de-p} + \left(\frac{2}{C(\de)}\right)^\frac{q}{\de-q} \left(\frac{1}{q}-\frac{1}{\de}\right) \big( \|\psi\|^{q-\de} \big)^\frac{\de}{\de-q} \\
	&\le \left(\frac{1}{q}-\frac{1}{\de}\right) \frac{2^{p/(\de-p)}}{\big(K_0 S_\de^{-\de}\big)^\ga} \frac{\|\psi\|^{-\de}}{C_\de^{\de/(\de-q)}},
	\end{align*}
	where we assumed $C_\de>1$ and $\big(K_0 S_\de^{-\de}\big)^\ga= \min\{ \big(K_0 S_\de^{-\de}\big)^p, \big(K_0 S_\de^{-\de}\big)^q  \}$.
	Therefore, from \eqref{eq39}, we observe that 
	\begin{align}\label{eq40}
	c\le \sup_{t\ge 0} \mc J(t\psi) \le \sup_{t\ge 0} h(t) \le \left(\frac{1}{q}-\frac{1}{\de}\right) \frac{2^{p/(\de-p)}}{\big(K_0 S_\de^{-\de}\big)^\ga} \frac{\|\psi\|^{-\de}}{C_\de^{\de/(\de-q)}}.
	\end{align}
	By (AR) and the fact that $\{u_n\}$ is a Cerami sequence, we get
	\begin{align*}
	c=\lim_{n\ra\infty} \Big(\mc J(u_n)-\frac{1}{\nu} \mc J^\prime(u_n)u_n \Big) &\ge \limsup_{n\ra \infty} \Big( \big(\frac{1}{p}-\frac{1}{\nu}\big) \|u_n\|_{s_1,p}^p + \big(\frac{1}{q}-\frac{1}{\nu}\big) \|u_n\|_{s_2,q}^q\Big)  \\
	&\ge \limsup_{n\ra \infty} \big(\frac{1}{p}-\frac{1}{\nu}\big) \|u_n\|_{s,\ga}^\ga,
	\end{align*}
	where $(s,\ga)\in\{ (s_1,p),(s_2,q)\}$. Then, using \eqref{eq40}, we obtain 
	\begin{align*}
	\limsup_{n\ra \infty} \|u_n\|_{s,\ga}^\ga\le \frac{p\nu}{\nu-p}c \le  \frac{p\nu}{\nu-p} \left(\frac{1}{q}-\frac{1}{\de}\right) \frac{2^{p/(\de-p)}}{\big(K_0 S_\de^{-\de}\big)^\ga} \frac{\|\psi\|^{-\de}}{C_\de^{\de/(\de-q)}}<\frac{1}{2^\ga},
	\end{align*}
	provided $C_\de$ is sufficiently large. Then, the proof of the lemma follows by using the definition of $\| u_n \|$.\QED
\end{proof}
In the subcritical case, we prove the boundedness of Cerami sequences. The proof differs from the critical case due to absence of Ambrosetti-Rabinowitz type condition for this case.
\begin{Lemma}
	Suppose that (f1)-(f3) hold. Then, any Cerami sequence of $\mc J$ at the mountain pass level $c$ is bounded. 
\end{Lemma}
\begin{proof}
	Let $\{v_n\}\subset X$ be a Cerami sequence of $\mc J$ at level $c$. Then, as in the proof of lemma \ref{lem2}, there exists $t_n\in[0,1]$ such that 
	\begin{align}\label{eq41}
	\mc J(t_nv_n)=\max_{t\in[0,1]} \mc J(tv_n). 
	\end{align}
	We claim that the sequence $\{\mc J(t_nv_n) \}$ is bounded.
	The claim is obvious if $t_n=0$ or $1$, therefore we assume $t_n\in(0,1)$. Also, we assume $v_n\ge 0$. Setting 
	\[ H(t):= tf(t)-pF(t)  \; \mbox{ for }t\in\mb R,\]
	and since $t^{1-p} f(t)$ is nondecreasing and differentiable (due to (f1) and (f3)), we get that $H$ is nondecreasing in $\mb R$. Now, from \eqref{eq41}, we have 
	\[ \frac{d}{dt} \mc J(tv_n)\Big|_{t=t_n}=0, \]
	and hence
	\begin{align*}
	p\mc J(t_nv_n) &= %\Big(\frac{r}{p}-1\Big) t_n^p \|v_n\|_{s_1,p}^p+
	\Big(\frac{p}{q}-1\Big) t_n^q \|v_n\|_{s_2,q}^q+ \int_{\mb R^N} \frac{K(x)H(t_nv_n)}{|x|^\ba} \\
	&\le %\Big(\frac{r}{p}-1\Big) \|v_n\|_{s_1,p}^p+ 
	\Big(\frac{p}{q}-1\Big)  \|v_n\|_{s_2,q}^q+ \int_{\mb R^N} \frac{K(x)H(v_n)}{|x|^\ba} \\
	&= p\mc J(v_n)- \mc J^\prime(v_n)v_n = pc+o_n(1),
	\end{align*}
	this proves the claim. To prove the lemma, on the contrary, we assume that up to a subsequence $\|v_n\| \ra \infty$ as $n\ra\infty$ and $\|v_n\| \ge 1$ for all $n\in\mb N$. Then, there exists $w\in X$ such that $w_n\rightharpoonup w$ weakly in $X$, where $w_n= \frac{v_n}{\|v_n\|}$. We claim that $w=0$ a.e. in $\mb R^N$. Indeed, since $\mc J(v_n)=c+o_n(1)$ and $\|v_n\| \ra \infty$, we get 
	\begin{align}\label{eq42}
	\frac{1}{p} \frac{\|v_n\|_{s_1,p}^p}{\|v_n\|^p} + \frac{1}{q} \frac{\|v_n\|_{s_2,q}^q}{\|v_n\|^p} -\int_{\mb R^N} \frac{K(x)F(v_n)}{\|v_n\|^p |x|^\ba}=o_n(1).
	\end{align}
	%  this implies 
	%  	  \begin{align}
	%  	  	\lim_{n\ra \infty} \int_{\mb R^N} \frac{K(x)F(v_n)w_n(x)^r}{|v_n(x)|^r |x|^\ba} \le 0.
	%  	  \end{align}
	Since $\lim_{t\ra\infty} t^{-p}F(t)=\infty$, for every $\tau>0$, there exists $\xi>0$ such that 
	\[ F(t)\ge \tau|t|^p \; \mbox{ for all }|t|\ge \xi.  \]
	Therefore, from \eqref{eq42} and noting that $q<p$, we obtain 
	\begin{align*}
	o_n(1)+\frac{1}{p} \ge \int_{|v_n|\ge \xi} \frac{K(x)F(v_n)w_n(x)^p}{|v_n(x)|^p |x|^\ba} \ge \tau \int_{\mb R^N} \frac{K(x)w_n(x)^p}{ |x|^\ba}\chi_{\{|v_n|\ge \frac{\xi}{\|v_n\|} \} }.
	\end{align*}
	By Fatou's lemma, for all $\tau>0$, we deduce that
	\[  \tau\int_{\mb R^N} \frac{K(x)w(x)^p}{ |x|^\ba}\le \frac{1}{p}, \]
	which implies $w=0$ a.e. in $\mb R^N$. Let $T>0$, then there exists $n_T\in\mb N$ such that for all $n\ge n_T$, $T\|v_n\|^{-1}\in (0,1)$. Now, from \eqref{eq41} and the fact that $\| w_n\|\le 1$ (follows from $w=0$), we get
	\begin{align*}
	\mc J(t_nv_n)\ge \mc J(Tw_n) &= \frac{T^p}{p} \|w_n\|_{s_1,p}^p + \frac{T^q}{q} \|w_n\|_{s_2,q}^q -\int_{\mb R^N} \frac{K(x)F(Tw_n)}{|x|^\ba} \\
	&\ge \frac{2^{1-p}T^\ga}{p} \|w_n\|^p -\int_{\mb R^N} \frac{K(x)F(Tw_n)}{|x|^\ba},
	\end{align*}
	where $T^\ga= \min\{ T^p,T^q\}$. Then, by the compactness lemma \ref{cmp}, we have 
	\begin{align*}
	\int_{\mb R^N} \frac{K(x)F(Tw_n)}{|x|^\ba}\ra 0, \; \mbox{ as }n\ra \infty.
	\end{align*}
	Thus,
	\begin{align*}
	\liminf_{n\ra\infty} \mc J(t_n v_n)\ge \frac{2^{1-p}T^\ga}{p},
	\end{align*}
	which is a contradiction, if we choose $T$ such that $ T= \big(2^p p \sup_n \{ \mc J(t_nv_n)\}\big)^{1/\ga}$.
	This proves the lemma. \QED
\end{proof}
\section{Proof of Main Theorem}
 The functional $\mc J$ satisfies mountain pass geometry in both the cases. Therefore, there exist Cerami sequences $\{u_n\}\subset X$ and $\{v_n\}\subset X$ in the subcritical and critical cases, respectively. Furthermore, $\{u_n\}$ and $\{v_n\}$ are bounded in $X$. Therefore, up to a subsequence $u_n\rightharpoonup u$ and $v_n\rightharpoonup v$ weakly in $X$, for some $u,v\in X$.
 
\subsection{The subcritical case}   By the compactness lemma \ref{cmp}, we see that $\int_{\mb R^N} \frac{K(x)f(u_n)}{|x|^\ba}(u_n-u)\ra 0$ as $n\ra\infty$.  Moreover, since $\langle \mc J^\prime(u_n), u_n-u \rangle \ra 0$ as $n\ra \infty$, it follows that
\begin{align*}
\mc A_1(u_n,u_n-u)+\mc A_2(u_n,u_n-u)+\int_{\mb R^N} V(x)\big(|u_n|^{p-2}u_n +|u_n|^{q-2}u_n\big)(u_n-u)=o_n(1).
\end{align*}
On the other hand for fixed $u\in X$, it is easy to observe that $\varTheta_{u,p}+\varTheta_{u,q}\in X^\prime$, where $\varTheta_{u,p}(v):= \mc A_1(u,v)+\int_{\mb R^N} V(x)|u|^{p-2}uvdx$ for all $v\in X$ and $\varTheta_{u,q}$ is analogously defined.  Therefore, using the fact that $u_n\rightharpoonup u$ weakly in $X$, we get
\begin{align*}
\mc A_1(u,u_n-u)+ \mc A_2(u,u_n-u)+\int_{\mb R^N} V(x)\big(|u|^{p-2}u +|u|^{q-2}u\big)(u_n-u)=o_n(1).
\end{align*}
Coupling these, we obtain
\begin{align}\label{eq9}
&\mc A_1(u_n,u_n-u)-\mc A_1(u,u_n-u)+\int_{\mb R^N} V(x)\big(|u_n|^{p-2}u_n -|u|^{p-2}u\big)(u_n-u) \nonumber \\
&\;+\mc A_2(u_n,u_n-u)- \mc A_2(u,u_n-u) +\int_{\mb R^N} V(x)\big(|u_n|^{q-2}u_n -|u|^{q-2}u\big)(u_n-u)=o_n(1).
\end{align}
Now, we consider the cases when $q\ge 2$ and $1<q<2$ (note that $p\ge 2$).\\
\textit{Case (i)}: $q\ge 2$.\\
Using the inequality  $|a-b|^{l} \leq 2^{l-2}(|a|^{l-2}a-|b|^{l-2}b)(a-b) \;\text{for}\; a, b \in \mathbb{R}^{n}\text{ and } l \geq 2,$ from \eqref{eq9}, it follows that
\begin{align*}
  [u_n-u]_{s_1,p}^p+\int_{\mb R^N} V(x)|u_n-u|^p + [u_n-u]_{s_2,q}^q+\int_{\mb R^N} V(x)|u_n-u|^q\le o_n(1),
\end{align*}
that is
\[ \|u_n-u\|_{s_1,p}^p+\|u_n-u\|_{s_2,q}^q =o_n(1) \]
this implies that $u_n\ra u$ in $X$.\\
\textit{Case (ii)}: $1<q<2$.\\
As we know that
for $a,\ b\in\mb{R}^n$ and $1<m<2$, there exists $C_m >0$  a constant such that
$$|a-b|^m\le C_m\big((|a|^{m-2}a-|b|^{m-2}b)(a-b)\big)^{\frac{m}{2}}(|a|^m+|b|^m)^{\frac{2-m}{2}}.$$
Set $a=u_k(x)-u_k(y)$, $b=u_\la(x)-u_\la(y)$ and then using H\"older inequality, we deduce that
\begin{align*}
[u_n-u]_{s_2,q}^q
&\le C \big(\mc A_2(u_n, u_n-u) - \mc A_2(u, u_n-u) \big)^\frac{q}{2}  \big( [u_n]_{s_2,q}^q + [u]_{s_2,q}^q \big)^\frac{2-q}{2}.
%&\\
%&\qquad\quad \left(\ds\int_{\mb R^{2N}} \frac{|u_n(x)-u_n(y)|^{q}+|u(x)-u(y)|^{q}}{|x-y|^{N+qs_2}} \right)^\frac{2-q}{2}
\end{align*}
and boundedness of $\{u_n\}$ in $X$, implies
\[ [u_n-u]_{s_2,q}^2\le C \big(\mc A_2(u_n, u_n-u) - \mc A_2(u, u_n-u) \big). \]
Therefore, using \eqref{eq9} and proceeding similarly as in the previous case, we obtain
%\begin{align*}
%  [u_n-u]_{s_1,p}^p+\int_{\mb R^N} V(x)|u_n-u|^p + [u_n-u]_{s_2,q}^2+\int_{\mb R^N} V(x)|u_n-u|^q = o_n(1),
%\end{align*}
$u_n\ra u$ in $\widetilde{W}^{s_1,p}_V(\mb R^N)$ as well as in $\widetilde{W}^{s_2,q}_V(\mb R^N)$, which gives us the required strong convergence of $u_n$ to $u$ in $X$. Using the fact that $c>0$ and strong convergence, we get that $u\not\equiv 0$. By Lemma \ref{non}, $u$ is a nontrivial nonnegative solution of $(\mc P)$. 

\subsection{The critical case}  We observe that if we choose $C_\de>0$ such that Lemma \ref{lem2} is satisfied, then the compactness results of corollary \ref{cor1} hold. Now, we can proceed similarly to prove that $v_n\ra v$ in $X$ and $v\not\equiv 0$, hence $v$ is a nontrivial weak solution of $(\mc P)$. \QED

\section*{Acknowledgments}
D. Kumar is thankful to Council of Scientific and Industrial Research (CSIR) for the financial support. K. Sreenadh acknowledges the support through the Project:
MATRICS grant MTR/2019/000121 funded by SERB, India.

\end{document}